\documentclass[11pt]{amsart}
\usepackage{amsmath, amssymb, latexsym}
\usepackage{mathrsfs}
\usepackage[all, knot]{xy}
\xyoption{arc}

\theoremstyle{definition}

\numberwithin{equation}{section}

\numberwithin{equation}{section}

\newtheorem{theorem}{Theorem}[section]

\newtheorem{lemma}[theorem]{Lemma}
\newtheorem{proposition}[theorem]{Proposition}

\theoremstyle{definition}
\newtheorem{definition}[theorem]{Definition}
\newtheorem{remark}[theorem]{Remark}

\newcommand{\HOM}{\text{HOM}}

\newcommand{\B}{\mathbb{B}}
\newcommand{\C}{\mathbb{C}}
\newcommand{\XX}{\mathbf{X}}
\newcommand{\YY}{\mathbf{Y}}
\newcommand{\HH}{\mathbf{H}}

\newcommand{\HHH}{\mathscr{H}}
\newcommand{\Q}{\mathbb{Q}}

\newcommand{\Z}{\mathbb{Z}}
\newcommand{\N}{\mathbb{N}}

\newcommand{\ii}{\textit{\textbf{i}}}
\newcommand{\jj}{\textit{\textbf{j}}}
\newcommand{\kk}{\textit{\textbf{k}}}
\newcommand{\Seq}{\text{Seq}}

\newcommand{\Pol}{\mathscr{P}}

\newcommand{\Ind}{\text{Ind}}
\newcommand{\Res}{\text{Res}}

\newcommand{\gdim}{\text{dim}^\pi_q}
\newcommand{\Mod}{\text{Mod}}
\newcommand{\pMod}{\text{pMod}}
\newcommand{\fMod}{\text{fMod}}

\newcommand{\re}{\text{re}}
\newcommand{\im}{\text{im}}

\newcommand{\od}{\overline{1}}
\newcommand{\ev}{\overline{0}}
\newcommand{\tx}{\text}

\newcommand{\genO}[3] 
{\fontsize{10}{10}\selectfont
\xy
(0,5.5)*{}; (0,-5.5)*{} **\dir{-};
(4,0)*{\cdots};
(8,5.5)*{}; (8,-5.5)*{} **\dir{-};
(12,0)*{\cdots};
(16,5.5)*{}; (16,-5.5)*{} **\dir{-};
(0,-8)*{#1}; (8,-8)*{#2}; (16,-8)*{#3};
\endxy\fontsize{11}{11}\selectfont}
\newcommand{\genX}[3] 
{\fontsize{10}{10}\selectfont
\xy
(0,5.5)*{}; (0,-5.5)*{} **\dir{-};
(4,0)*{\cdots};
(8,5.5)*{}; (8,-5.5)*{} **\dir{-};
(8,0)*{\bullet}; (12,0)*{\cdots};
(16,5.5)*{}; (16,-5.5)*{} **\dir{-};
(0,-8)*{#1}; (8,-8)*{#2}; (16,-8)*{#3};
\endxy\fontsize{11}{11}\selectfont}
\newcommand{\genT}[4] 
{\fontsize{10}{10}\selectfont
\xy
(0,5.5)*{}; (0,-5.5)*{} **\dir{-};
(3.5,0)*{\cdots};
(5.5,5.5)*{}; (12.5,-5.5)*{} **\dir{-};
(12.5,5.5)*{}; (5.5,-5.5)*{} **\dir{-};
(14.5,0)*{\cdots};
(18,5.5)*{}; (18,-5.5)*{} **\dir{-};
(0,-8)*{#1}; (5.5,-8)*{#2}; (12.5,-8)*{#3}; (18,-8)*{#4};\endxy
\fontsize{11}{11}\selectfont}

\newcommand{\SSL}[2]
{\fontsize{10}{10}\selectfont
\xy
(0,5)*{}; (0,-5)*{} **\dir{-}\POS?(.25)="x";
(3.5,0)*{\cdots};
(7,5)*{}; (7,-5)*{} **\dir{-}\POS?(.75)="y";
"x"*{\bullet};"y"*{\bullet};(0,-7)*{#1}; (7,-7)*{#2};
\endxy\fontsize{11}{11}\selectfont}
\newcommand{\SSR}[2]
{\fontsize{10}{10}\selectfont
\xy
(0,5)*{}; (0,-5)*{} **\dir{-}\POS?(.75)="x";
(3.5,0)*{\cdots};
(7,5)*{}; (7,-5)*{} **\dir{-}\POS?(.25)="y";
"x"*{\bullet};"y"*{\bullet};(0,-7)*{#1}; (7,-7)*{#2};
\endxy\fontsize{11}{11}\selectfont}
\newcommand{\SCL}[3]
{\fontsize{10}{10}\selectfont
\xy
(0,5)*{}; (0,-5)*{} **\dir{-}\POS?(.25)="x";
(3.5,0)*{\cdots};
(7,5)*{}; (7,0)*{} **\dir{-};
(11,5)*{}; (11,0)*{} **\dir{-};
(7,0)*{}; (11,-5)*{} **\dir{-};
(11,0)*{}; (7,-5)*{} **\dir{-};
"x"*{\bullet};(0,-7)*{#1}; (7,-7)*{#2};(11,-7)*{#3};
\endxy\fontsize{11}{11}\selectfont}
\newcommand{\SCR}[3]
{\fontsize{10}{10}\selectfont
\xy
(0,5)*{}; (0,-5)*{} **\dir{-}\POS?(.75)="x";
(3.5,0)*{\cdots};
(7,5)*{}; (11,0)*{} **\dir{-};
(11,5)*{}; (7,0)*{} **\dir{-};
(7,0)*{}; (7,-5)*{} **\dir{-};
(11,0)*{}; (11,-5)*{} **\dir{-};
"x"*{\bullet};(0,-7)*{#1}; (7,-7)*{#2};(11,-7)*{#3};
\endxy\fontsize{11}{11}\selectfont}
\newcommand{\SCLL}[3]
{\fontsize{10}{10}\selectfont
\xy
(-4,5)*{}; (0,0)*{} **\dir{-};
(0,5)*{}; (-4,0)*{} **\dir{-};
(-4,0)*{}; (-4,-5)*{} **\dir{-};
(0,0)*{}; (0,-5)*{} **\dir{-};
(3.5,0)*{\cdots};
(7,5)*{}; (7,-5)*{} **\dir{-}\POS?(.75)="x";
"x"*{\bullet};(-4,-7)*{#1};(0,-7)*{#2}; (7,-7)*{#3};
\endxy\fontsize{11}{11}\selectfont}
\newcommand{\SCRR}[3]
{\fontsize{10}{10}\selectfont
\xy
(-4,-5)*{}; (0,0)*{} **\dir{-};
(0,-5)*{}; (-4,0)*{} **\dir{-};
(-4,0)*{}; (-4,5)*{} **\dir{-};
(0,0)*{}; (0,5)*{} **\dir{-};
(3.5,0)*{\cdots};
(7,5)*{}; (7,-5)*{} **\dir{-}\POS?(.25)="x";
"x"*{\bullet};(-4,-7)*{#1};(0,-7)*{#2}; (7,-7)*{#3};
\endxy\fontsize{11}{11}\selectfont}
\newcommand{\CCL}[4]
{\fontsize{10}{10}\selectfont
\xy
(-4,5)*{}; (0,0)*{} **\dir{-};
(0,5)*{}; (-4,0)*{} **\dir{-};
(-4,0)*{}; (-4,-5)*{} **\dir{-};
(0,0)*{}; (0,-5)*{} **\dir{-};
(3.5,0)*{\cdots};
(7,5)*{}; (7,0)*{} **\dir{-};
(11,5)*{}; (11,0)*{} **\dir{-};
(7,0)*{}; (11,-5)*{} **\dir{-};
(11,0)*{}; (7,-5)*{} **\dir{-};
(-4,-7)*{#1};(0,-7)*{#2}; (7,-7)*{#3};(11,-7)*{#4};
\endxy\fontsize{11}{11}\selectfont}
\newcommand{\CCR}[4]
{\fontsize{10}{10}\selectfont
\xy
(-4,5)*{}; (-4,0)*{} **\dir{-};
(0,5)*{}; (0,0)*{} **\dir{-};
(-4,0)*{}; (0,-5)*{} **\dir{-};
(0,0)*{}; (-4,-5)*{} **\dir{-};
(3.5,0)*{\cdots};
(7,5)*{}; (11,0)*{} **\dir{-};
(11,5)*{}; (7,0)*{} **\dir{-};
(7,0)*{}; (7,-5)*{} **\dir{-};
(11,0)*{}; (11,-5)*{} **\dir{-};
(-4,-7)*{#1};(0,-7)*{#2}; (7,-7)*{#3};(11,-7)*{#4};
\endxy\fontsize{11}{11}\selectfont}

\newcommand{\DC}[2]
{\fontsize{10}{10}\selectfont
\xy
(0,6)*{}="T1"; (6,6)*{}="T2";
(0,-6)*{}="B1"; (6,-6)*{}="B2";
"T1"; "B1" **\crv{(9, 0)};
"T2"; "B2" **\crv{(-3,0)};
(0,-8)*{#1}; (6,-8)*{#2};
\endxy\fontsize{11}{11}\selectfont}
\newcommand{\DCL}[2]
{ \fontsize{10}{10}\selectfont
\xy
(0,4)*{}; (0,-4)*{} **\dir{-};
(5,4)*{}; (5,-4)*{} **\dir{-};
(0,0)*{\bullet}; (0,-6)*{#1}; (5,-6)*{#2};
\endxy\fontsize{11}{11}\selectfont}
\newcommand{\DCR}[2]
{\fontsize{10}{10}\selectfont
\xy
(0,4)*{}; (0,-4)*{} **\dir{-};
(5,4)*{}; (5,-4)*{} **\dir{-};
(5,0)*{\bullet}; (0,-6)*{#1}; (5,-6)*{#2};
\endxy\fontsize{11}{11}\selectfont}

\newcommand{\AS}[2]
{\fontsize{10}{10}\selectfont
\xy
(0,5)*{}; (0,-5)*{} **\dir{-};
(6,5)*{}; (6,-5)*{} **\dir{-};
(0,-7)*{#1}; (6,-7)*{#2};
\endxy\fontsize{11}{11}\selectfont}
\newcommand{\LU}[3]
{\fontsize{10}{10}\selectfont
\xy
(0,4.5)*{}; (7,-4.5)*{} **\dir{-} \POS?(.25)="x";
(7,4.5)*{}; (0,-4.5)*{} **\dir{-};
"x"*{\bullet}; "x"+(-2,2)*{#1}; (0,-6.5)*{#2}; (7,-6.5)*{#3}; \endxy\fontsize{11}{11}\selectfont}
\newcommand{\RD}[3]
{\fontsize{10}{10}\selectfont
\xy
(0,4.5)*{}; (7,-4.5)*{} **\dir{-} \POS?(.75)="x";
(7,4.5)*{}; (0,-4.5)*{} **\dir{-};
"x"*{\bullet}; "x"+(2,2)*{#1}; (0,-6.5)*{#2}; (7,-6.5)*{#3}; \endxy\fontsize{11}{11}\selectfont}
\newcommand{\RU}[3]
{\fontsize{10}{10}\selectfont
\xy
(0,4.5)*{}; (7,-4.5)*{} **\dir{-};
(7,4.5)*{}; (0,-4.5)*{} **\dir{-} \POS?(.25)="x";
"x"*{\bullet}; "x"+(2,2)*{#1}; (0,-6.5)*{#2}; (7,-6.5)*{#3}; \endxy\fontsize{11}{11}\selectfont}
\newcommand{\LD}[3]
{\fontsize{10}{10}\selectfont
\xy
(0,4.5)*{}; (7,-4.5)*{} **\dir{-};
(7,4.5)*{}; (0,-4.5)*{} **\dir{-} \POS?(.75)="x";
"x"*{\bullet}; "x"+(-2,2)*{#1}; (0,-6.5)*{#2}; (7,-6.5)*{#3}; \endxy\fontsize{11}{11}\selectfont}

\newcommand{\BraidL}[3]
{\fontsize{10}{10}\selectfont
\xy
(0,6)*{}="T1"; (5,6)*{}="T2";  (10,6)*{}="T3";
(0,-6)*{}="B1"; (5,-6)*{}="B2"; (10,-6)*{}="B3";
"T1"; "B3" **\crv{(0,0)&(10,0)}; "T3"; "B1" **\crv{(10,0)&(0,0)};
"T2"; "B2" **\crv{(5,5)&(3,4)&(0,1.5)&(0,-1.5)&(3,-4)&(5,-5)};
(0,-8)*{#1}; (5,-8)*{#2};(10,-8)*{#3};
\endxy\fontsize{11}{11}\selectfont}
\newcommand{\BraidR}[3]
{\fontsize{10}{10}\selectfont
\xy
(0,6)*{}="T1"; (5,6)*{}="T2";  (10,6)*{}="T3";
(0,-6)*{}="B1"; (5,-6)*{}="B2"; (10,-6)*{}="B3";
"T1"; "B3" **\crv{(0,0)&(10,0)}; "T3"; "B1" **\crv{(10,0)&(0,0)};
"T2"; "B2" **\crv{(5,5)&(7,4)&(10,1.5)&(10,-1.5)&(7,-4)&(5,-5)};
(0,-8)*{#1}; (5,-8)*{#2};(10,-8)*{#3};
\endxy\fontsize{11}{11}\selectfont}
\newcommand{\threeA}[3]
{ \fontsize{9}{9}\selectfont
\xy
(0,4)*{}; (0,-4)*{} **\dir{-};
(4,4)*{}; (4,-4)*{} **\dir{-};
(8,4)*{}; (8,-4)*{} **\dir{-};
(0,0)*{\bullet}; (0,-6)*{#1}; (4,-6)*{#2};(8,-6)*{#3};
\endxy\fontsize{11}{11}\selectfont}
\newcommand{\threeB}[3]
{ \fontsize{9}{9}\selectfont
\xy
(0,4)*{}; (0,-4)*{} **\dir{-};
(4,4)*{}; (4,-4)*{} **\dir{-};
(8,4)*{}; (8,-4)*{} **\dir{-};
(4,0)*{\bullet};(0,-6)*{#1}; (4,-6)*{#2};(8,-6)*{#3};
\endxy\fontsize{11}{11}\selectfont}
\newcommand{\threeC}[3]
{ \fontsize{9}{9}\selectfont
\xy
(0,4)*{}; (0,-4)*{} **\dir{-};
(4,4)*{}; (4,-4)*{} **\dir{-};
(8,4)*{}; (8,-4)*{} **\dir{-};
(8,0)*{\bullet};(0,-6)*{#1}; (4,-6)*{#2};(8,-6)*{#3};
\endxy\fontsize{11}{11}\selectfont}

\title[]
{Quiver Hecke algebras for Borcherds-Cartan datum III:
Categorification of quantum Borcherds superalgebras}

\author[Wan Wu]{Wan Wu}
\address{Harbin Engineering University,
Harbin, China}
\email{wuwan1818@163.com}



\begin{document}
\maketitle
\begin{abstract}
We introduce a family of the quiver Hecke superalgebras which give a categorification of quantum Borcherds superalgebras.
\end{abstract}

\let\cleardoublepage\clearpage\tableofcontents

\section*{\textbf{Introduction}}

\vskip 2mm

The quiver Hecke superalgebra is the superalgebraic counterpart of the Khovanov-Lauda-Rouquier (KLR) algebra (also called the quiver Hecke algebra) \cite{KL2009,KL2011,Rou}. First introduced by Kang, Kashiwara, and Tsuchioka in \cite{KKT2016}, it was explicitly constructed to rigorously link affine Hecke-Clifford superalgebras and affine Sergeev superalgebras. This extension has further furnished a profound theoretical framework for studying the modular representation theory of spin symmetric groups \cite{KL2022,K2025}.

For the one-vertex quiver, the quiver Hecke superalgebra reduces to the odd nil-Hecke algebra $ONH_n$—the odd analogue of the nil-Hecke algebra, first introduced by Ellis, Khovanov and Lauda in \cite{EKL2014}.
Concretely, this algebra can be realized as a matrix algebra over $O\Lambda_n$, the odd symmetric polynomial algebra. 
Both $O\Lambda_n$ and its direct limit $O\Lambda$ (the ring of odd symmetric functions) are fundamental, as they are the odd analogues of symmetric functions. 
Paralleling the indispensable role of symmetric functions in geometry and combinatorics, these odd symmetric algebraic objects also hold essential applications in the same fields \cite{E2013,EK2012,LR2014}.

In contrast to the well-established result that KLR algebras categorify the (half) quantum group, Hill and Wang \cite{HW2015} exploited the quiver Hecke superalgebra to categorify a covering algebra of both the half quantum group and the half quantum superalgebra. This covering algebra is defined as a  $\Q(q)^\pi$-algebra,  which specializes to a quantum Kac-Moody algebra when $\pi=1$ and to a quantum Kac-Moody superalgebra when $\pi=-1$. 
The most notable innovation lies in the identification of the parameter $\pi$ with the parity shift action on superalgebras and supermodules during the categorification procedure. 
Specifically, they proved that the Grothendieck group of the category of projective graded supermodules over the quiver Hecke superalgebra is isomorphic to the aforementioned covering algebra as  $\Q(q)^\pi$-algebras. 
In particular, the specialization $\pi=-1$ yields a categorification of quantum Kac-Moody superalgebras.

 A parallel result was established by Kang, Kashiwara, and Oh \cite{KKO2013,KKO2014}, who employed this algebraic framework to furnish a supercategorification of quantum groups and further to provide a categorification of irreducible highest-weight modules. In both cases, their approach hinges on utilizing the cyclotomic quiver Hecke superalgebras to categorify irreducible highest-weight modules, following the theoretical framework established in \cite{KK2012}.
 
Hill and Wang \cite{HW2015} proposed that their methodological framework is applicable to the quantum Borcherds algebra setting. This implies the potential construction of a quiver Hecke superalgebra associated with a Borcherds-Cartan datum, which would categorify both the covering algebra of the quantum Borcherds algebra (as defined in Section 1.3) and the quantum Borcherds superalgebra (as defined in \cite{BKM1998}).

Notably, the construction of such a quiver Hecke superalgebra depends on the prior establishment of the corresponding quiver Hecke algebra, while the categorification of the aforementioned covering algebra relies on the categorification of the conventional quantum group. In previous work, Tong and Wu \cite{TW2023, TW2025} achieved the categorification of quantum Borcherds algebras (also called quantum generalized Kac–Moody algebras) associated with arbitrary Borcherds-Cartan data—this breakthrough facilitates the construction of the superalgebraic counterpart (i.e., the quiver Hecke superalgebra) and thereby enables the categorification of both the quantum Borcherds superalgebra and its covering algebra.


In this paper, we extend the construction introduced in \cite{TW2023} to the context of Borcherds Cartan superdata, and apply the methodology established in \cite{HW2015} to categorify the covering algebras of quantum Borcherds algebras and quantum Borcherds superalgebras. Given an arbitrary Borcherds–Cartan superdatum, consisting of a $\Z_2$-graded indexed $I=I_{\ev}\sqcup I_{\od}$ and a  matrix $\widetilde A$ parametrized by an $I$ with diagonal can be $\leq 0$, we define the associated covering quantum Borcherds algebra $\HH$, which is a $\Q(q)^\pi$ twisted bialgebra. The specializations of $\pi=1$ and $\pi=-1$ of $\HH$ coincide with the negative parts of the quantum Borcherds algebra $\XX$ and the negative part of the quantum Borcherds superalgebra $\YY$, respectively. Furthermore, we establish that $\HH$ can be realized as the quotient algebra $\HHH/\text{rad}\{\ ,\ \}_\pi$, where $\text{rad}\{\ ,\ \}_\pi$ denotes the radical of the Lusztig form $\{\ ,\ \}_\pi$ on the free algebra $\HHH$.

We next construct the quiver Hecke superalgebras $R(\nu) (\nu\in\N[I])$ and derive their polynomial representations via braid-like planar diagrams, building on the foundational work of \cite{TW2023} and  \cite{TW2025}.
We further consider the degenerate case $\nu=ni$ for fixed $i\in I$, where only the subcase $i\in I^{\im}\cap I_{\od}$ necessitates novel analysis. For $i\in I^{\re}$, $R(ni)$ reduces to the (odd) nil-Hecke algebra, 
and the case $i\in I^{\im}\cap I_{\ev}$ has been addressed in \cite{TW2023}.
 
 Finally, we proceed to characterize the type $\mathtt M$ phenomenon of $R(\nu)$  via the Kashiwara operator construction in \cite{TW2025}, a cornerstone result for categorification and related topics. 
 Let $K_0(R(\nu))$ be the Grothendieck group of finitely generated projective graded $R(\nu)$-supermodules. 
 The aggregate Grothendieck group $K_0(R)=\bigoplus_{\nu\in \N[I]}K_0(R(\nu))$ carries a twisted bialgebra structure induced by module induction and restriction.
 We show a twisted $\Z^\pi[q,q^{-1}]$-bialgebra isomorphism between $K_0(R)$ and $_{\mathcal A^\pi}\HH$ (the $\mathcal A^\pi$ form of  $\HH$).
 Specializing $\pi$ to $-1$ thereby yields a categorification of the quantum Borcherds superalgebra as Theorem \ref{them:iso-}.



\vskip 6mm

\section{\textbf{Preliminaries}}

 Let $\pi$ be an indeterminate. For any ring $R$, define $R^\pi=R[\pi]/(\pi^2-1)$.

\subsection{Graded superalgebras}\

\vskip 2mm

A \emph{graded superalgebra}
 $A=\bigoplus_{n\in\Z,\varepsilon\in\Z_2}A_n^{\varepsilon}$ is a $(\Z,\Z_2)$-graded algebra over  $\C$. For a bi-homogeneous element $a\in A_n^{\varepsilon}$, we write $|a|=n$ for its degree and set

 $$p(a)=\begin{cases}0\ \ & \tx{if}\ \varepsilon=\ev, \\ 1 & \tx{if}\ \varepsilon=\od. \end{cases}$$
 Given two graded superalgebras $A$ and $B$, their tensor product $A\otimes B$ is a graded superalgebra with the multiplication given by
  $$(a\otimes b)(a'\otimes b')=(-1)^{p(b)p(a')}aa'\otimes bb'.$$

 Let $A,B$ be graded superalgebras. A \emph{graded $A$-$B$-bisupermodule} $M=\bigoplus_{n\in\Z, \varepsilon\in\Z_2}M_n^{\varepsilon}$ is a $(\Z,\Z_2)$-graded $A$-$B$-bimodule.
 For two such bisupermodules $M$ and $N$, set $$\HOM_{A\tx{-}B}(M,N)=\bigoplus_{n\in\Z, \varepsilon\in\Z_2}\HOM_{A\tx{-}B}(M,N)_n^\varepsilon,$$ where $\HOM_{A\tx{-}B}(M,N)_n^\varepsilon$ consists of all linear maps $f:M\rightarrow N$ of bidegree $(n,\varepsilon)$ such that
 $$f(amb)=(-1)^{p(f)p(a)}af(m)b \quad \tx{for}\ a\in A, b\in B\ \tx{and}\ m\in M.$$
 By omitting  $B$ or $A$ we obtain the notions of left or right graded supermodules and  their morphisms. Unless otherwise stated, all graded supermodules are assumed to be left modules.

 For a graded superspace $V$, i.e., a $(\Z,\Z_2)$-graded vector space, we denote by $|V|$ its underlying  $\Z$-graded vector space  (forgetting parity). For graded
 $A$-supermodules $M$ and $N$, we write $M\simeq N$ if there exists an isomorphism in $\HOM_{A}(M,N)_0^{\ev}$, and write $M\cong N$ if there exists an isomorphism in $\HOM_{A}(M,N)_0 $, that is, if $|M|$ and $|N|$ are isomorphic as  graded $|A|$-modules.

 Let $M$ be a  graded $A$-supermodule and let $q$ be an indeterminate.  If $M_n:=M_n^{\ev}\oplus M_n^{\od} =0$ for $n\ll 0$, we define
 $$\gdim M:=\sum_{n\in\Z}\left(\text{dim}M_n^{\ev}+\pi\text{dim}M_n^{\od}\right)q^n\in\Z^\pi((q)).$$
For $k\in \Z$, the \emph{degree shift} $M\{k\}$ is the graded supermodule obtained from $M$ by putting $(M\{k\})_n^{\varepsilon}=M_{n-k}^{\varepsilon}$.  The \emph{parity shift}  $\Pi M$ is obtained  by putting $(\Pi M)_n^{\varepsilon}=M_n^{\varepsilon+\od}$ and with a new action $a\circ m=(-1)^{p(a)}am$.  We have $$\HOM_A(\Pi^rM\{k\},N)\simeq \Pi^r\HOM_A(M,N)\{-k\}$$ as $(\Z,\Z_2)$-graded vector spaces.

For $f(q,\pi)=\sum_{n\in \Z}(a_n+\pi b_n)q^n\in\N^\pi[q,q^{-1}]$, define a graded supermodule $$f(q,\pi) M:=\bigoplus_{n\in\Z}\left((M\{n\})^{\oplus a_n}\oplus(\Pi M\{n\})^{\oplus b_n}\right),$$
so that $\gdim f(q,\pi)M =f(q,\pi)\cdot\gdim M$.

Given  graded $A$-supermodule $M$ and graded $B$-supermodule $N$, their tensor product  is a graded $A\otimes B$-supermodule with action
 $$a\otimes b\cdot(m\otimes n)=(-1)^{p(b)p(m)} am\otimes bn.$$
 If $f\in\HOM_{A}(M,M')$ and $g\in\HOM_{B}(N,N')$, then $f\otimes g\in\HOM_{A\otimes B}(M\otimes N, M'\otimes N')$ acts by
  $$f\otimes g(m\otimes n)=(-1)^{p(g)p(m)} f(m)\otimes g(n).$$

 An irreducible graded $A$-supermodule $V$ is said\ to be of \emph{type $\mathtt{M}$} if $|V|$ is an irreducible graded $|A|$-module; otherwise $V$ is of \emph{type $\mathtt{Q}$}.
A graded superalgebra $A=\bigoplus_{n\in\Z}A_n$  is \emph{Laurentian} if each $A_n$ is finite-dimensional and $A_n = 0$ for $n \ll 0$. 

\begin{lemma}\label{irr super}{\it  Assume $A$ is Laurentian.
\begin{itemize}
\item [(1)] Up to degree and parity shifts, there are only finitely many irreducible graded $A$-supermodules, all of which are finite-dimensional.
\item [(2)] Let $V$ be an  irreducible graded $A$-supermodule. Then $V$ is of type $\mathtt Q$ if and only if $\Pi V\simeq V$. Moreover, if $V$ is of type $\mathtt Q$, then $|V|\cong N\oplus N^\varsigma$ for some irreducible graded $|A|$-modules $N$  with $N\ncong N^\varsigma$. \\
    Here, $N^\varsigma$ is the graded $|A|$-module obtained from $N$ by twisting the action via $a\ast n=(-1)^{p(a)}an$. We refer to $N$ as $V^+$ and to $N^\varsigma$ as $V^-$.
\item [(3)] The graded Jacoboson radical $J^{gr}(|A|)$ of $|A|$ coincides with the graded super Jacoboson radical of $A$: $$J^{gr\tx{-}super}(A)=\bigcap\{\tx{ann}(V)\mid V \ \tx{irreducible graded} \ A\tx{-supermodule}\}.$$
\item [(4)] Suppose that, up to degree and parity shifts, $V_1,\dots, V_r$ of type $\mathtt M$ and $V_{r+1},\dots, V_t$ of type $\mathtt Q$ is a non-redundant set of irreducible graded $A$-supermodules, then
    \begin{itemize}
\item [(i)] up to degree  shifts, $V_1,\dots, V_r, \Pi V_1,\dots, \Pi V_{r}, V_{r+1},\dots, V_t$ is   a  non-redundant set of irreducible graded $A$-supermodules,
\item [(ii)] up to degree  shifts, $V_1,\dots, V_r,   V_{r+1}^\pm,\dots, V_t^\pm$ is  a non-redundant set of irreducible graded $|A|$-modules.
\end{itemize}
\end{itemize}
}\end{lemma}
\begin{proof}
Part (1) is a direct consequence of \cite[Theorem 2.7.2]{NV2004}.
The statements in  (2)-(4) follow from the corresponding results in \cite[Section 12.2]{K2005}.
\end{proof}

\subsection{Quantum Borcherds (super)algebras}\

\vskip 2mm

Let $I=I_{\ev}\sqcup I_{\od}$ be a $\Z_2$-graded index set. A {\it Borcherds-Cartan superdatum} $(I,\widetilde A,\cdot)$ consists of
\begin{itemize}
\item [(1)] an integer-valued matrix $\widetilde A=(a_{ij})_{i,j\in I}$ satisfying
\begin{itemize}
\item [(i)] $a_{ii}=2,0,-2,-4,\dots,$
\item [(ii)] $a_{ij}\in \Z_{\leq 0}$ for $i\neq j$,
\item [(iii)] $a_{ij}=0$ if and only if $a_{ji}=0$,
\item [(iv)] $a_{ij}\in 2\Z$ for $i\in I_{\od},j\in I$,
\item [(v)] there is a diagonal matrix $D=\text{diag}(r_i\in\Z_{>0}\mid i\in I)$ such that $D\widetilde A$ is symmetric.
\end{itemize}
\item [(2)] a symmetric bilinear form $\nu,\nu'\mapsto \nu\cdot \nu'$ on $\Z[I]$ taking values in $\Z$, such that $$i\cdot j=r_ia_{ij}=r_ja_{ji} \ \ \text{for all} \ i,j\in I.$$
\end{itemize}

We set $I^\re=\{i\in I\mid a_{ii}=2\}$ and $I^\im=\{i\in I\mid a_{ii}\leq 0\}$. For each $i\in I$, define $q_i=q^{r_i}$.  For $i\in I^\re, n\in \N$,  set
 $$[n]_i^\pi=\frac{(\pi^{p(i)}q_i)^n-q_i^{-n}}{\pi^{p(i)}q_i-q_i^{-1}},\ \ [n]_i^\pi !=[n]_i^\pi [n-1]_i^\pi \cdots [1]_i^\pi. $$
In particular, we write
 $$[n]_i^+=[n]_i^\pi|_{\pi=1},\ \ [n]_i^-=[n]_i^\pi|_{\pi=-1}.$$

\begin{definition}The (half part) {\it quantum Borcherds algebra} $\XX$ associated to  $(I,\widetilde A,\cdot)$ is the $\Q(q)$-algebra generated by $f_i$ $(i\in I)$, subject to
\begin{equation*}
\begin{aligned}
& \sum_{a+b=1-a_{ij}}(-1)^af_i^{(a)}f_jf_i^{(b)}=0 \quad \text{for} \ i\in I^\re,j\in I\ \text{and}\ i\neq j,\\
& f_if_j-f_jf_i=0 \quad \text{for}\ i,j\in I\ \text{with}\ i\cdot j=0.
\end{aligned}
\end{equation*}
Here  $f_i^{(a)}=f_i^a/[a]^+_i!$. The algebra $\XX$ is $\N[I]$-graded by assigning $|f_i|=i$.
\end{definition}

 By \cite[Chapter 1]{Lus} and \cite[Section 3]{Kas91}, $\XX$ admits a comultiplication
 $$\rho_+:\XX\rightarrow \XX\otimes \XX,\ \ f_i\mapsto f_i\otimes 1+1\otimes f_i\ \ \tx{for}\ i\in I,$$
 where $\XX\otimes \XX$ is endowed with the twisted multiplication
$$(x_1\otimes x_2)(y_1\otimes y_2)=q^{-|x_2|\cdot |y_1|}x_1y_1\otimes x_2y_2.$$ There is a nondegenerate symmetric bilinear form $\{ \ , \ \}_+: \XX\times \XX\rightarrow \Q(q)$ satisfying
\begin{itemize}
\item[(i)] $\{x, y\}_+ =0$ if $|x| \neq |y|$,
\item[(ii)] $\{1,1\}_+ = 1$,
\item[(iii)] $\{f_{i}, f_{i}\}_+ = (1-q_i^2)^{-1}$ for all $i\in I$,
\item[(iv)] $\{x, yz\}_+ = \{\rho_+(x), y \otimes z\}_+$  for $x,y,z
\in \XX$.
\end{itemize}

Let $\mathcal A=\Z[q,q^{-1}]$ be the ring of Laurent polynomials. The $\mathcal A$-form ${_{\mathcal A} \XX}$ is the $\mathcal A$-subalgebra of $\XX$ generated by the  $f_i^{(n)}$ for $i\in I^\re, n\in \Z_{>0}$ and $f_i$ for $i\in I^\im$.

\begin{definition}\cite[Definition 2.7]{BKM1998}\   The (half part) {\it quantum Borcherds superalgebra} $\YY$ associated to  $(I,\widetilde A,\cdot)$ is the $\Q(q)$-algebra generated by $f_i$ $(i\in I)$, subject to
\begin{equation*}
\begin{aligned}
& \sum_{a+b=1-a_{ij}}(-1)^{a+p(a;i,j)}f_i^{(a)}f_jf_i^{(b)}=0 \quad \text{for} \ i\in I^\re,j\in I\ \text{and}\ i\neq j,\\
& f_if_j-(-1)^{p(i)p(j)}f_jf_i=0 \quad \text{for}\ i,j\in I\ \text{with}\ i\cdot j=0.
\end{aligned}
\end{equation*}
Here   $p(a;i,j)=ap(i)p(j)+\frac{1}{2}a(a-1)p(i)$ and $f_i^{(a)}=f_i^a/[a]^-_i!$.
The algebra $\YY$ is $(\N[I],\Z_2)$-graded by assigning $|f_i|=i$ and $p(f_i)=p(i)$.
\end{definition}

Again, by \cite[Chapter 1]{Lus} and \cite[Section 3]{Kas91}, $\YY$ admits a comultiplication
 $$\rho_-:\YY\rightarrow \YY\otimes \YY,\ f_i\mapsto f_i\otimes 1+1\otimes f_i\ \ \tx{for}\ i\in I,$$
  where $\YY\otimes\YY$ is endowed with the twisted multiplication
$$(x_1\otimes x_2)(y_1\otimes y_2)=(-1)^{p(x_2)p(y_1)}q^{-|x_2|\cdot |y_1|}x_1y_1\otimes x_2y_2.$$ One can also show that there is a nondegenerate symmetric bilinear form $\{ \ , \ \}_-: \YY\times \YY\rightarrow \Q(q)$ satisfying
\begin{itemize}
\item[(i)] $\{x, y\}_- =0$ if $|x| \neq |y|$,
\item[(ii)] $\{1,1\}_- = 1$,
\item[(iii)] $\{f_{i}, f_{i}\}_- = (1-(-1)^{p(i)}q_i^2)^{-1}$ for all $i\in I$,
\item[(iv)] $\{x, yz\}_- = \{\rho_-(x), y \otimes z\}_-$  for $x,y,z
\in \YY$.
\end{itemize}
A detailed illustration of these results is provided in the Appendix, where the same arguments also apply to the quantum Borcherds algebra.

Similarly, the $\mathcal A$-form ${_{\mathcal A} \YY}$ is the $\mathcal A$-subalgebra of $\YY$ generated by the  $f_i^{(n)}$ for $i\in I^\re,n\in \Z_{> 0}$ and $f_i$ for $i\in I^\im$.

\subsection{Covering algebras}\

\vskip 2mm

Given a Borcherds-Cartan superdatum $(I,\widetilde{A},\cdot)$, we define a covering algebra $\HH$ for $\XX$ and $\YY$. This construction was first introduced in \cite{HW2015}.

Let $\HHH$ be the free $\Q(q)^\pi$-algebra generated by $\theta_i$ $(i\in I)$ with $|\theta_i|=i$ and $p(\theta_i)=p(i)$. We have a comultiplication $\rho_\pi:\HHH\rightarrow \HHH\otimes \HHH$ given by $\rho_\pi{(\theta_i)}=\theta_i\otimes 1+1\otimes \theta_i$ ($i\in I$). Here $\HHH\otimes \HHH$ is endowed with the twisted multiplication
$$(x_1\otimes x_2)(y_1\otimes y_2)=\pi^{p(x_2)p(y_1)}q^{-|x_2|\cdot |y_1|}x_1y_1\otimes x_2y_2.$$ As in \cite[1.2.3]{Lus}, there is a   symmetric bilinear form $\{ \ , \ \}_\pi:\HHH\times \HHH\rightarrow \Q(q)^\pi$ satisfying
\begin{itemize}
\item[(i)] $\{x, y\}_\pi =0$ if $|x| \neq |y|$,
\item[(ii)] $\{1,1\}_\pi= 1$,
\item[(iii)] $\{\theta_{i}, \theta_{i}\}_\pi = (1-\pi^{p(i)}q_i^2)^{-1}$ for all $i\in I$,
\item[(iv)] $\{x, yz\}_\pi= \{\rho_\pi(x), y \otimes z\}_\pi$  for $x,y,z
\in \HHH$.
\end{itemize}
Similar to \cite[1.4.5]{Lus}, the following elements lie in the radical of the bilinear form  $\{ \ , \ \}_\pi$
\begin{equation}\label{serrepp}
\begin{aligned}
& \sum_{a+b=1-a_{ij}}(-1)^{a}\pi^{p(a;i,j)}\theta_i^{(a)}\theta_j\theta_i^{(b)} \quad \text{for} \ i\in I^\re,j\in I\ \text{and}\ i\neq j,\\
& \theta_i\theta_j-\pi^{p(i)p(j)}\theta_j\theta_i \quad \text{for}\ i,j\in I\ \text{with}\ i\cdot j=0.
\end{aligned}
\end{equation}
Here  $\theta_i^{(a)}=\theta_i^a/[a]^\pi_i!$. More generally, the radical also contains the elements
$$\sum_{a+b=1-na_{ij}}(-1)^{a}\pi^{p(a;i,j;n)}\theta_i^{(a)}\theta_j^n\theta_i^{(b)}  \quad \text{for} \ i\in I^\re,j\in I,n\geq 1\ \text{and}\ i\neq j,$$
where $p(a;i,j;n)=anp(i)p(j)+\frac{1}{2}a(a-1)p(i)$ (see, for instance, \cite[Theorem~1.4]{FKKT2022}).

Let $\HH$ be the quotient of $\HHH$ by the two-sided ideal generated by (\ref{serrepp}).

\begin{proposition}\label{nondeg} {\it $\HH\cong \HHH/\text{rad}\{ \ , \ \}_\pi.$}
\begin{proof}There are canonical $\Q(q)$-algebra isomorphisms, given by $\theta_i\mapsto f_i$ $(i\in I)$:
\begin{align*} \HH\big/(\pi-1)\cong \Q(q)\otimes_{\pi=1}\HH \cong \XX,\\
\HH\big/(\pi+1)\cong \Q(q)\otimes_{\pi=-1}\HH \cong \YY.\end{align*}
The tensor products here are taken over $\Q(q)^\pi$, with $\Q(q)$ viewed as a $\Q(q)^\pi$-module via the specialization $\pi\mapsto 1$ (resp. $\pi\mapsto -1$). Under the isomorphism, the bilinear form $\{ \ , \ \}_\pi $ specialized at $\pi=1$ (resp. $\pi=-1$) on $\HH\big/(\pi-1)$ (resp. $\HH\big/(\pi+1)$) coincides with $\{ \ , \ \}_+$ (resp. $\{ \ , \ \}_-$), and is therefore nondegenerate in both cases. Now, let $x$ lie in the radical of  $\{ \ , \ \}_\pi$ on $\HH$. Then $x\in (\pi-1)\cap (\pi+1)=(\pi^2-1)=0$, which implies that
  $\{ \ , \ \}_\pi$ is nondegenerate on $\HH$.
Consequently, we have $\HH\cong \HHH/\text{rad}\{ \ , \ \}_\pi$.
\end{proof}
\end{proposition}

 The $\mathcal A^\pi$-form ${_{\mathcal A^\pi} \HH}$ is the $\mathcal A^\pi$-subalgebra of $\HH$ generated by the  $\theta_i^{(n)}$ for $i\in I^\re,n\in \Z_{> 0}$ and $\theta_i$ for $i\in I^\im$. Note that
 \begin{equation}\label{A-form}{_{\mathcal A^\pi}\HH}\big/(\pi-1)\cong   {_{\mathcal A}\XX},\quad
{_{\mathcal A^\pi}\HH}\big/(\pi+1)\cong   {_{\mathcal A}\YY}. \end{equation}

\vskip 5mm

\section{\textbf{Quiver Hecke superalgebras for Borcherds-Cartan superdatum}}

Given a Borcherds-Cartan superdatum $(I,\widetilde{A},\cdot)$, for $i,j\in I$ with $i\neq j$, we define $$T_{ij}=\{(a,b)\in \N\times \N \mid r_i a+r_j b=-i\cdot j;\ a\in 2\Z \ \text{if}\ i\in I_{\od};\ b\in 2\Z \ \text{if}\ j\in I_{\od}\}.$$
Choose a family of integers $\{t_{i,j;a,b}\}_{i\neq j,(a,b)\in T_{ij}}$ such that
$t_{i,j;a,b}=t_{j,i;b,a}$ and $t_{i,j;-a_{ij},0}\neq 0$. For each pair $i\neq j$ ,
define the polynomial
$$Q_{ij}(u,v)=\sum_{(a,b)\in T_{ij}}t_{i,j;a,b}\ u^av^b\in\Z\left<u,v\right>/ \langle uv-(-1)^{p(i)p(j)}vu \rangle.$$
This can be written explicitly as
$$Q_{ij}(u,v)=t_{i,j;-a_{ij},0}u^{-a_{ij}}+\sum_{a<-a_{ij},b<-a_{ji}}t_{i,j;a,b}\ u^av^b+t_{i,j;0,-a_{ji}}v^{-a_{ji}}.$$

By construction, we have $Q_{ij}(u,v)=Q_{ji}(v,u)$. Assume that $|u|=2r_i, p(u)=p(i)$ and $|v|=2r_j, p(v)=p(j)$. Then  $Q_{ij}(u,v)$  is homogeneous of degree $-2 i\cdot j$ and is parity even.

If $i\cdot j=0$, we take $Q_{ij}(u,v)=1$  for simplicity.

\subsection{Quiver Hecke superalgebras}\

\vskip 2mm

Fix $\nu=\sum_{i\in I}\nu_i i\in \N[I]$ with $\text{ht}{(\nu)}:=\sum_{i\in I}\nu_i=n$. Let $\text{Seq}(\nu)$ be the set of all sequences $\ii=i_1i_2\dots i_n$ in $I$ such that $\nu=i_1+i_2\cdots +i_n$.  
For $\ii \in \text{Seq}(\nu)$, the graded superalgebra $R(\nu)$ associated with  $(\widetilde A, I, \cdot)$ is generated by the bihomogeneous planar diagrams:
$$\begin{aligned}
&\hspace{3mm} 1_{\ii}=\genO{i_1}{i_k}{i_n} \quad  \text{with} \ |1_{\ii}|=0,\ p(1_{\ii})=0, \\ \\
& x_{k,\ii}=\genX{i_1}{i_k}{i_n} \quad 1\leq k\leq n,\  \text{with} \  |x_{k,\ii}|=2r_{i_k},\ p(x_{k,\ii})=p(i_k),\\ \\
& \tau_{k,\ii}=\genT{i_1}{i_k}{i_{k+1}}{i_n} \quad  1\leq k\leq n-1,\  \text{with} \ |\tau_{k,\ii}|=-i_k\cdot i_{k+1},\ p(\tau_{k,\ii})=p(i_k)p(i_{k+1}).
\end{aligned}$$

\vspace{2mm}

\noindent Subject to the following local relations as in the tensor product of superalgebras

\begin{equation*}
\SSL{i}{j}\ =\ (-1)^{p(i)p(j)}\ \SSR{i}{j},\quad\quad\quad \CCL{i}{j}{k}{l}\ =\ (-1)^{p(i)p(j)p(k)p(l)}\ \CCR{i}{j}{k}{l},
\end{equation*}

$$ \SCL{i}{j}{k}\ =\ (-1)^{p(i)p(j)p(k)}\ \SCR{i}{j}{k}, \quad\quad\quad  \SCLL{i}{j}{k}\ =\ (-1)^{p(i)p(j)p(k)}\ \SCRR{i}{j}{k}, $$
and the local relations
\begin{equation}
 \DC{i}{j} \  =  \  \begin{cases}
 \  0 & \text{ if } i= j, \\  \\
\ Q_{ij}\Big(\ \DCL{i}{j} \ , \ \DCR{i}{j}\ \Big) & \text{ if } i\neq  j,
  \end{cases}
\end{equation}

\begin{equation}
 \LU{{}}{i}{i} \   -  \ \RD{{}}{i}{i}\ =\  \LD{{}}{i}{i}   \ - \ \RU{{}}{i}{i}\ =\ \AS{i}{i}  \quad \text{ if } i \in I^{\tx{re}}\cap I_{\ev},
\end{equation}

\begin{equation}
 \LU{{}}{i}{i} \   +  \ \RD{{}}{i}{i}\ =\  \LD{{}}{i}{i}   \ + \ \RU{{}}{i}{i}\ =\ \AS{i}{i}  \quad \text{ if } i \in I^{\tx{re}}\cap I_{\od},
\end{equation}

\begin{equation}
 \LU{{}}{i}{j} \   = \ (-1)^{p(i)p(j)} \RD{{}}{i}{j}  \quad\quad\quad \LD{{}}{i}{j}   \ = \ (-1)^{p(i)p(j)} \RU{{}}{i}{j}  \quad \text{ otherwise},
\end{equation}

\begin{equation}\label{braid}
 \BraidL{i}{j}{i} \  -  \ \BraidR{i}{j}{i} \ = {\large\  \begin{cases} \frac{ Q_{ij}(u,v)-Q_{ij}(w,v)} { u-w}  & \text{ if } i\in I^{\tx{re}}\cap I_{\ev},\ i\ne j,\vspace{3mm} \\
  (-1)^{p(j)}(u-w)\frac{ Q_{ij}(u,v)-Q_{ij}(w,v)} { u^2-w^2}  & \text{ if } i\in I^{\tx{re}}\cap I_{\od},\ i\ne j.
 \end{cases}}
 \end{equation}

 \quad Here $u=\threeA{i}{j}{i},\ v=\threeB{i}{j}{i},\ w= \threeC{i}{j}{i}$
 \begin{equation}
\BraidL{i}{j}{k} \  = \ \BraidR{i}{j}{k} \quad\quad \text{otherwise}.
\end{equation}

\vskip 2mm

For a detailed explanation of the braid-like planar diagrams and their multiplication, we refer to \cite{KL2009}. For $\ii,\jj\in \Seq(\nu)$,  set $_{\jj}R(\nu)_{\ii}=1_{\jj}R(\nu)1_{\ii}$. Then, $R(\nu)$ decomposes as
$$R(\nu)=\bigoplus_{\ii,\jj\in\Seq(\nu)} {_{\jj}R(\nu)_{\ii}}.$$

\subsection{Polynomial representations}\

\vskip 2mm
 Let $\mathscr{P}$ be the superalgebra with even generators $y_1,\dots,y_n,z_1,\dots,z_n$  and odd generators  $c_1,\dots, c_n$, subject to the following relations:
\begin{itemize}
\item[(i)] $y_iy_j=y_jy_i,\  z_iz_j=z_jz_i, \ y_iz_j=z_jy_i$ $(1\leq i,j\leq n)$,
\item[(ii)] $c_i^2=1$,\ $c_ic_j=-c_jc_i$  $(1\leq i\neq j\leq n)$,
\item[(iii)] $c_iy_j=(-1)^{\delta_{ij}}y_jc_i,\ c_iz_j=(-1)^{\delta_{ij}} z_jc_i$   $(1\leq i,j \leq n)$.
\end{itemize}

The symmetric group $S_n$ acts on $\mathscr{P}$  by permuting the symbols $y_i,z_j$, and $c_k$ independently. For $1\leq k\leq n-1$, define a linear operator $\sigma_k$   on $\mathscr{P}$ by
\begin{align*} & \sigma_k(y_k)=-1-c_kc_{k+1},\  \sigma_k(y_{k+1})=1-c_kc_{k+1},\  \sigma_k(y_{j})=0\  \text{for}\ j \neq k,k+1,\\
& \sigma_k(z_j)=\sigma_k(c_j)=0\ \ \text{for all}\ j,
\end{align*}
 and the Leibniz rule
 $$\sigma_k(fg)=\sigma_k(f)g+s_k(f)\sigma_k(g)\ \ \tx{for all}\ f,g\in\mathscr{P}.$$

By \cite[Lemma 3.6]{Le2023}, these operators satisfy the nilCoxeter relations
 $$\sigma_k^2=0,\ \sigma_k\sigma_l=\sigma_l\sigma_k \ (|k-l|>1),\ \sigma_k\sigma_{k+1}\sigma_k=\sigma_{k+1}\sigma_k\sigma_{k+1}.$$
Similarly, we define $\sigma_k'$ to be the operator obtained by interchanging  $y_j$ with $z_j$ in the definition of $\sigma_k$.

\begin{remark}\label{sig} By definition,  for $f=f(y_1,\cdots,y_n)\in \C[y_1,\dots,y_n]\subset\mathscr{P}$, we have
$$\sigma_k(f)=\frac{s_kf-f}{y_k-y_{k+1}}+c_kc_{k+1}\frac{\overline{s}_kf-f}{y_k+y_{k+1}},$$
where $\overline{s}_kf=f(y_1,\dots,y_{k-1},-y_{k+1},-y_k,y_{k+2}\dots,y_n)$.
\end{remark}

For each $\ii\in \Seq(\nu)$, we define a superalgebra $\mathscr{P}_{\ii}$ with even generators $y_1(\ii),\dots,y_n(\ii)$, $z_1(\ii),\dots,z_n(\ii)$  and odd generators  $c_1(\ii),\dots, c_n(\ii)$, subject to
\begin{itemize}
\item[(i)] $y_i(\ii)y_j(\ii)=y_j(\ii)y_i(\ii),\  z_i(\ii)z_j(\ii)=z_j(\ii)z_i(\ii), \ y_i(\ii)z_j(\ii)=z_j(\ii)y_i(\ii)$ $(1\leq i,j\leq n)$,
\item[(ii)] $c_i(\ii)^2=1$,\ $c_i(\ii)c_j(\ii)=-c_j(\ii)c_i(\ii)$  $(1\leq i\neq j \leq n)$,
\item[(iii)] $c_k(\ii)y_l(\ii)= (-1)^{\delta_{kl}p(i_k)}y_l(\ii)c_k(\ii),\ c_k(\ii)z_l(\ii)=(-1)^{\delta_{kl}p(i_k)}z_l(\ii)c_k(\ii)$   $(1\leq k,l\leq n)$.
\end{itemize}
 and form the $\C$-vector space
$\mathscr{P}_{\nu}=\bigoplus_{\ii\in \Seq(\nu)}\mathscr{P}_{\ii}$.
Let $s_k$  acts on $\mathscr{P}_{\nu}$   by
$$  y_j(\ii)\mapsto y_{s_k (j)}(s_k\ii),\ z_j(\ii)\mapsto z_{s_k (j)}(s_k\ii),\ c_j(\ii)\mapsto c_{s_k (j)}(s_k\ii)\ \ \text{for}\ 1\leq j\leq n. $$
 If $i_k=i_{k+1}\in  I_{\ev}$, we define  algebraic operators $\widetilde s_k$ and $\overset{\approx}{s}_k$ on $\mathscr{P}_{\ii}$ by
$$ \widetilde s_k:\ y_j(\ii)\mapsto y_{s_k (j)}(\ii),\ z_j(\ii)\mapsto z_{j}(\ii),\ c_j(\ii)\mapsto c_j(\ii)\ \ \text{for}\ 1\leq j\leq n, $$
$$ \overset{\approx}{s}_k:\ y_j(\ii)\mapsto y_{s_k(j)}(\ii),\ z_j(\ii)\mapsto z_{s_k (j)}(\ii),\ c_j(\ii)\mapsto c_j(\ii)\ \ \text{for}\ 1\leq j\leq n. $$

Let $\Lambda$ be a graph with vertex set $I$, where there is an oriented edge between vertices $i$ and $j$ if $i\neq j$.  Choose a collection $\{\gamma_{ij}\}_{i\neq j\in I}\subset\C$ such that $$\gamma_{ij}\gamma_{ji}=-1/2\ \ \tx{if}\ i,j\in I_{\od};\ \   \gamma_{ij}=1\ \ \tx{if}\ i\in I_{\ev} \  \tx{or}\  j\in I_{\ev}.$$

\begin{definition}We define an action of   $R({\nu})$ on $\Pol_{\nu}$ as follows:

\begin{itemize}

\item[(i)]  If $\ii\neq\kk$,  $_{\jj}R(\nu)_{\ii}$ acts on $\Pol_{\kk}$ by $0$.

\vskip 1mm

\item[(ii)]  For $f\in \Pol_{\ii}$,
$1_{\ii}\cdot f=f,  \ x_{k,\ii}\cdot f=c_{k}(\ii)^{p(i_k)}y_k{(\ii)}f.$

\vskip 1mm

\item[(iii)] Assume $\ii=i_1\dots i_n$ with $i_k=i, i_{k+1}=j$. For $f\in \Pol_{\ii}$,
\begin{equation*}
\tau_{k,\ii}\cdot f=\begin{cases}
\frac{f-\widetilde s_kf}{y_k{(\ii)}-y_{k+1}(\ii)} & \text{if}\ i=j\in I^\re\cap I_{\ev},\vspace{1mm}\\
\frac{\widetilde s_kf- \overset{\approx}{s}_k f}{z_k{(\ii)}-z_{k+1}(\ii)} & \text{if}\ i=j\in I^\im\cap I_{\ev},\vspace{1mm}\\
\frac{1}{2}(c_k(\ii)-c_{k+1}(\ii))\sigma_kf & \text{if}\ i=j\in I^\re\cap I_{\od},\vspace{1mm}\\
\frac{1}{2}(c_k(\ii)-c_{k+1}(\ii))\sigma_k'f & \text{if}\ i=j\in I^\im\cap I_{\od},\vspace{1mm}\\
\frac{1}{\gamma_{ij}}\left(\frac{1}{2}(c_k(s_k\ii)-c_{k+1}(s_k\ii))\right)^{p(i)p(j)}s_kf & \text{if}\  i\leftarrow j,\vspace{1mm}\\
\frac{1}{\gamma_{ij}}(\frac{1}{2}\left(c_k(s_k\ii)-c_{k+1}(s_k\ii))\right)^{p(i)p(j)}\left(Q_{ij}(x_{k+1,s_k\ii},x_{k,s_k\ii})\cdot s_kf\right)\ & \text{if}\  i  \rightarrow j.
\end{cases}
\end{equation*}
\end{itemize}
\end{definition}

\vskip 2mm

\begin{proposition}\label{P}{\it
$\Pol_\nu$ is a $R(\nu)$-module with the action defined above.}
\begin{proof}
We check the defining relations for $R(\nu)$.
We check only the local relation (\ref{braid}) in the case where $i\in I^{\text{re}}\cap I_{\od},j\in I_{\od}$ and $i \rightarrow j$.  The remaining cases are straightforward to verify.
For simplicity, we omit the label $\ii$ and take $k=1$. Note that
\begin{align*}
 \tau_1\tau_2\tau_1\cdot f=\frac{1}{8\gamma_{ij}\gamma_{ji}}(c_1-c_2)(c_1-c_3)(c_2-c_3)\ s_1\sigma_2\big(Q_{ij}(x_2,x_1)s_1f\big)\\
=\frac{1}{2}(c_3-c_1) \Big[\big(s_1\sigma_2(Q_{ij}(x_2,x_1)\big)f+  Q_{ij}(x_3,x_2)s_1\sigma_2s_1f\Big].
\end{align*}
According to Remark \ref{sig},
\begin{align*}
s_1\sigma_2(Q_{ij}(x_2,x_1))=\frac{Q_{ij}(x_3,x_2)-Q_{ij}(x_1,x_2)}{y_1-y_{3}}+c_1c_{3}\frac{Q_{ij}(x_3,x_2)-Q_{ij}(x_1,x_2)}{y_1+y_3}\\
=(y_1+y_3)\frac{Q_{ij}(x_3,x_2)-Q_{ij}(x_1,x_2)}{x_3^2-x_1^2}+c_1c_{3}(y_1-y_3)\frac{Q_{ij}(x_3,x_2)-Q_{ij}(x_1,x_2)}{x_3^2-x_1^2}.
\end{align*}
Thus, $$
\tau_1\tau_2\tau_1\cdot f=(x_3-x_1)\frac{Q_{ij}(x_3,x_2)-Q_{ij}(x_1,x_2)}{x_3^2-x_1^2} f+ \frac{1}{2}(c_3-c_1) Q_{ij}(x_3,x_2)s_1\sigma_2s_1f. $$
On the other hand, we have
 $$\tau_2\tau_1\tau_2\cdot f=\frac{1}{2}(c_3-c_1)  Q_{ij}(x_3,x_2)s_2\sigma_1s_2f. $$
It suffices to verify that $s_1\sigma_2s_1f =s_2\sigma_1s_2f$, which can be proven by induction.
\end{proof}
\end{proposition}

\subsection{Algebras $R(ni)$ }\

\vskip 2mm

Let $\nu = ni$ for some $i \in I$. We abbreviate the generators of $R(ni)$ by $x_1, \dots, x_n$ and $\tau_1, \dots, \tau_{n-1}$. For each $\omega \in S_n$, fix a reduced expression $\omega = s_{k_1} \cdots s_{k_t}$  and define  $\tau_\omega= \tau_{k_1}\cdots \tau_{k_t}\in R(ni)$. Note that $\tau_\omega$ is independent of the choice of reduced expression, up to a sign $\pm 1$, when $i$ is odd.

\textbf{Case 1}: $i  \in I_{\ev}$.

 The  algebra $R(ni)$ is  the nil-Hecke algebra $NH_n$ when $i\in I^\re$ (see \cite[Example 2.2]{KL2009}),  and the algebra described in \cite[Section 2.2]{TW2023} when $i\in I^\im$. These algebras have a  basis $\{x_1^{u_1}\cdots x_n^{u_n}\tau_\omega\mid \omega\in S_n,u_1,\dots,u_n\geq 0\}$.
The polynomial algebra $P_n = \mathbb{C}[x_1, \dots, x_n]$ can be identified with the subalgebra of $R(ni)$ generated by $x_1, \dots, x_n$. The center $Z(R(ni))$ consists of the symmetric polynomials in $x_1, \dots, x_n$.

 If $i\in I^\re$, up to isomorphism and degree shifts, $NH_n$ has a unique irreducible graded module $V(i^n)$, which may have graded dimension   $[n]_i^+!$. Its graded projective cover is given by
\begin{equation}\label{peven}P_{i^{(n)}}= NH_ne_{i,n}\left\{\binom{n}{2} \cdot r_i\right\},\end{equation}
 where $e_{i,n} = x_1^{n-1} x_2^{n-2} \cdots x_{n-1} \tau_{\omega_0}$, and $\omega_0$ is the longest element in $S_n$.

 If $i\in I^\im$, then the one-dimensional trivial  module is the unique irreducible graded $R(ni)$-module.

\textbf{Case 2}:  $i \in I^\re \cap I_{\od}$.

The  superalgebra $R(ni)$ is  the odd nil-Hecke algebra $ONH_n$ considered in \cite{EKL2014}, which has a basis $\{x_1^{u_1}\cdots x_n^{u_n}\tau_\omega\mid \omega\in S_n,u_1,\dots,u_n\geq 0\}$.
The odd polynomial algebra $OP_n = \mathbb{C}\left<x_1, \dots, x_n\right>/\left<x_ix_j=-x_jx_i,i\neq j\right>$ can be identified with the subalgebra of $ONH_n$ generated by $x_1, \dots, x_n$. For $n\geq 2$, the center $Z(ONH_n)$ consists of the symmetric polynomials in $x_1^2, \dots, x_n^2$.

Up to isomorphism and degree/parity   shifts, $ONH_n$ has a unique irreducible graded supermodule $V(i^n)$, which may have $(q,\pi)$-dimension   $[n]_i^\pi!$. Its  projective cover is given by \begin{equation}\label{podd}P_{i^{(n)}}= ONH_ne_{i,n}\left\{\binom{n}{2} \cdot r_i\right\},\end{equation}
 where $e_{i,n} = (-1)^{\binom{n}{3}}x_1^{n-1} x_2^{n-2} \cdots x_{n-1} \tau_{\omega_0}$ for the certain reduced expression
\begin{equation}\label{om0}\omega_0=s_1(s_2s_1)\cdots(s_{n-1}\cdots s_1).\end{equation}
Since $e_{i,n}$ is homogeneous in parity, the simple module $V(i^n)$ is of type $\mathtt M$.

\begin{remark}
Note that   $\mathscr{P}_{ni}=\mathscr P$ in this case. We identify $OP_n$ with the subalgebra of $\mathscr P$ generated by $c_1y_1,\dots,c_ny_n$, via the correspondence $x_j\mapsto c_jy_j$.  Under this identification, $ONH_n$ acts on $OP_n$  by $x_k\cdot f=x_kf$ and   $\tau_k\cdot f=\frac{1}{2}(c_k -c_{k+1}) \sigma_kf$.  Thus,
\begin{align*}
& \tau_k(x_k)=\tau_k(x_{k+1})=1,\ \tau_k(x_j)=0\ \text{for}\ j \neq k,k+1,\\
& \tau_k (fg)=1/2(c_k -c_{k+1})\sigma_k(fg)=1/2(c_k -c_{k+1})\big(\sigma_k(f)g+s_k(f)\sigma_k(g)\big)\\
& \phantom{\tau_k (fg)}=\tau_k (f)g+s_k^-(f)\tau_k(g),
\end{align*}
where $ {s}_k^-f(x_1,\dots,x_n)=f(-x_1,\dots,-x_{k+1},-x_k,\dots,-x_n)$. Therefore, $ONH_n$ acts on $OP_n$  via the odd divided difference operators defined in  \cite[2.1.1]{EKL2014}. In particular,  let $\omega_0$ be as in (\ref{om0}), then by \cite[Proposition 3.6]{EKL2014}, we have
$$\tau_{\omega_0}(x_1^{n-1} x_2^{n-2} \cdots x_{n-1})=(-1)^{\binom{n}{3}}.$$
\end{remark}

\textbf{Case 3}: $i \in I^\im \cap I_{\od}$.

\begin{proposition}\label{S2}{\it
$R(ni)$ has a basis $\{x_1^{u_1}\cdots x_n^{u_n}\tau_\omega\mid \omega\in S_n,u_1,\dots,u_n\geq 0\}$.}
\begin{proof}
We show that these elements act on $\Pol_{ni}$ $(=\mathscr P)$ linearly independently. Suppose that we have a non-trivial linear combination
$\sum_{\omega;u_1,\dots,u_n}k_{\omega;u_1,\dots,u_n}x_1^{u_1}\cdots x_n^{u_n}\tau_\omega$ acts on $\Pol_{ni}$ by zero. Choose an element $\omega$ of minimal length  such that $k_{\omega;u_1,\dots,u_n}\neq 0$ for some $u_1,\dots,u_n$. Write $\omega_0=\omega\omega'$, and apply the linear combination to the element
$$\tau_{\omega'}\cdot\big((c_1z_1)^{n-1}(c_2z_2)^{n-2}\cdots (c_{n-1}z_{n-1})\big).$$
This yields
$$ \textstyle{\sum_{u_1,\dots,u_n}}\pm k_{\omega;u_1,\dots,u_n}x_1^{u_1}\cdots x_n^{u_n}=0,$$
which implies $k_{\omega;u_1,\dots,u_n}=0$ for all $u_1,\dots,u_n\geq 0$, a contradiction.
\end{proof}
\end{proposition}

\begin{proposition}\label{S3}{\it
For $n\geq 2$, the center $Z(R(ni))=\C[x_1^2,\dots,x_n^2]^{S_n}$.}
\begin{proof}
It is obvious that $\C[x_1^2,\dots,x_n^2]^{S_n}\subset Z(R(ni))$.  Let $z=\sum_{\omega\in S_n}f_\omega\tau_\omega$ be a center element. Assume that $\omega\neq 1$ with $f_\omega\neq 0$, then there exists $k\in \{1,\dots,n\}$ such that $\omega (k)\neq k$. But this implies $x_kz-zx_k=\sum_{\omega\in S_n}f_\omega(x_k-(-1)^{l(\omega)}x_{\omega(k)})\tau_\omega\neq 0$. Thus $z\in OP_n$. Since $Z(OP_n)=\C[x_1^2,\dots,x_n^2]$, we write $z=\sum_{a,b\geq 0}h_{ab}x_1^{2a}x_2^{2b}$ for some $h_{ab}\in \C[x_3^2,\dots,x_n^2]$.  Now $\tau_1z=z\tau_1$ implies $p_{ab}=p_{ba}$ for each $a,b$. Hence $z$ is symmetric in $x_1$ and $x_2$. Similarly, we can show that $z$ is symmetric in $x_k$ and $x_{k+1}$ for all $1\leq k\leq n-1$.
\end{proof}
\end{proposition}

If $i\cdot i<0$, then $R(ni)$  has a unique maximal left graded superideal, namely
 $R(ni)_{>0}$. Consequently, it has a unique irreducible graded supermodule $V(i^n)$, which is the one-dimensional trivial module. It is clear  that $V(i^n)$ is of type $\mathtt M$   as a supermodule.

If $i\cdot i=0$, then the unique maximal left graded superideal of $R(ni)$ is spanned by $$  \{x_1^{u_1}\cdots x_n^{u_n}\tau_\omega\mid \omega\in S_n,u_1,\dots,u_n\geq 0\}\setminus \{1\}. $$
Hence the one-dimensional trivial module is again the unique irreducible graded supermodule; we denote it by $V(i^n)$ as well.

\subsection{Basis and center of $R(\nu)$}\

\vskip 2mm

Let $\ii,\jj\in \nu$. Using the polynomial representation $\mathscr P_{\nu}$ of $R(\nu)$, one can  obtain by a similar argument in \cite[Theorem 2.5]{KL2009} that $\mathscr P_{\nu}$ is a faithful  $R(\nu)$-module and  $_{\ii}R(\nu)_{\jj}$ has a basis $$\{x_{1,\ii}^{u_1}\cdots x_{n,\ii}^{u_n}\cdot\widehat\omega_\jj\mid u_1,\dots,u_n\geq 0,\ \omega\in S_n\ \text{such that} \ \omega(\jj)=\ii\},$$ where $\widehat\omega_\jj\in {_{\ii}} R(\nu)_{\jj}$ is uniquely determined by a fixed  reduced expression of $\omega$.  If the reduced expression for $\omega$ is $s_{k_1}\cdots s_{k_r}$, then $$\widehat\omega_\jj=\tau_{k_1,s_{k_2}\cdots s_{k_r}(\jj)}\cdots \tau_{k_{r-1},s_{k_r}(\jj)}\tau_{k_r,\jj}.$$

Assume $\Seq(\nu)$ contains a sequence ${i_1}^{m_1}\cdots{i_t}^{m_t}$ such that ${i_1},\dots,{i_t}$ are all distinct. Similar to \cite[Theorem 2.9]{KL2009}, the center $Z(R{(\nu)})$ can be  described as $$Z(R(\nu))\cong \bigotimes^t_{k=1}\C[\chi_{1}^{1+p(i_k)},\dots,\chi_{m_k}^{1+p(i_k)}]^{S_{m_k}}.$$

\vskip 5mm

\section{\textbf{Categorification of quantum Borcherds superalgebras}}

\subsection{Grothendieck groups of graded $R(\nu)$-supermodules}\

\vskip 2mm

Let $R(\nu)$-$\Mod$ denote the category of finitely generated graded $R(\nu)$-supermodules, whose morphisms are the degree-zero even homomorphisms:
$$\HOM_{R(\nu)}(M,N)_0^{\ev}\ \ \tx{for}\ M,N\in R(\nu)\tx{-}\Mod.$$

Let $R(\nu)$-$\fMod$ (resp. $R(\nu)$-$\pMod$) be the full subcategory of $R(\nu)$-$\Mod$ consisting  of finite-dimensional (resp. finitely generated projective) graded $R(\nu)$-supermodules.

Since $R(\nu)$  is Laurentian, Lemma \ref{irr super}(1) implies that there are only finitely many irreducible graded $R(\nu)$-supermodules up to isomorphism and degree/parity shifts. All such modules are finite-dimensional, and they remain irreducible as $R(\nu)$-supermodules after forgetting the $\Z$-grading.

Let $\B_{\nu}$ be the set of equivalence classes (under isomorphism and degree/parity   shifts) of  irreducible graded $R(\nu)$-supermodules. Choose one representative $S_b$ for
each equivalence class and denote by $P_b$ the projective cover of $S_b$.

The Grothendieck group $G_0(R(\nu))$ (resp. $K_0(R(\nu))$) of $R(\nu)$-$\fMod$ (resp. $R(\nu)$-$\pMod$) are free $\Z^\pi[q,q^{-1}]$-modules with a basis $\{[S_b]\}_{b\in \B_{\nu}}$ (resp. $\{[P_b]\}_{b\in \B_{\nu}}$) (see Remark \ref{type} below). Here the actions of $q$ and $\pi$ are given by
$$q[M]=[M\{1\}],\ \ \pi[M]=[\Pi M].$$
Let $R=\bigoplus_{\nu\in \N[I]}R(\nu)$  and form
$$G_0(R)=\bigoplus_{\nu\in \N[I]}G_0(R(\nu)),\ \ K_0(R)=\bigoplus_{\nu\in \N[I]}K_0(R(\nu)).$$

\begin{remark}\label{type} ({\bf Type $\mathtt M$ phenomenon}) \
It's necessary to show that all the irreducible graded $R(\nu)$-supermodules are type $\mathtt M$. Otherwise, the Grothendieck groups $G_0(R)$ and $K_0(R)$ are not free $\Z^\pi[q,q^{-1}]$-modules; in particular, one cannot naturally identify $K_0(R\otimes R)$  with $K_0(R)\otimes K_0(R)$, nor $G_0(R\otimes R)$ with $G_0(R)\otimes G_0(R)$  (see \cite[Lemma 12.2.13]{K2005}).

By Lemma \ref{irr super}(4), this is equivalent to proving that every irreducible graded $|R|$-module $V$ satisfies $V\cong V^\varsigma$.   An effective method to establish this is provided in \cite[Theorem 6.4]{KKO2013}. Using the framework developed in \cite[Section 2]{TW2025}, one can construct the Kashiwara operators on the category of irreducible graded $|R|$-modules and then applies \cite[Theorem~6.4]{KKO2013} to conclude the claim.
\end{remark}

The $\Z^\pi[q,q^{-1}]$-module $K_0(R)$ (resp. $G_0(R)$) is a twisted bialgebra whose product and coproduct are induced by the induction and restriction functors:
\begin{equation*}
\begin{aligned}
& \Ind^{\nu+\nu'}_{\nu,\nu'}\colon R(\nu)\otimes R(\nu')\text{-}\Mod\rightarrow R(\nu+\nu')\text{-}\Mod,\ M\mapsto R(\nu+\nu')1_{\nu,\nu'}\otimes_{R(\nu)\otimes R(\nu')}M,\\
&\Res^{\nu+\nu'}_{\nu,\nu'}\colon R(\nu+\nu')\text{-}\Mod\rightarrow R(\nu)\otimes R(\nu')\text{-}\Mod,\ N\mapsto 1_{\nu,\nu'}N,
\end{aligned}
\end{equation*}
where $1_{\nu,\nu'}=1_\nu\otimes 1_{\nu'}$.

More precisely, assume $\ii=i_1\dots i_n\in\Seq(\nu)$.  For $x\in R(\nu)\text{-}\Mod$, we set $$|x|:=\nu \in \N[I],\quad p(x)=p(\nu):=p(i_1)+\cdots+p(i_n).$$
Due to Remark \ref{type}, we may identify $K_0(R\otimes R)$ (resp. $G_0(R\otimes R)$) with $K_0(R)\otimes K_0(R)$ (resp. $G_0(R)\otimes G_0(R)$).  We equip $K_0(R)\otimes K_0(R)$ (resp. $G_0(R)\otimes G_0(R)$) with the twisted algebra structure given by
$$(x_1\otimes x_2)(y_1\otimes y_2)=\pi^{p(x_2)p(y_1)}q^{-|x_2|\cdot |y_1|}x_1y_1\otimes x_2y_2,$$
With this structure, the restriction map $\Res$ becomes an algebra homomorphism by the following Mackey-type theorem.

\begin{proposition}\label{Mackey}\cite[Proposition 2.18]{KL2009}, \cite[Theorem 6.3]{HW2015} \ {\it
Let $\nu,\nu',\mu,\mu'\in \N[I]$ with $\nu+\nu'=\mu+\mu'$. For $M\in R(\mu)\text{-}\Mod, N\in R(\mu')\text{-}\Mod$, we have a filtration of  $\Res_{\nu,\nu'}\Ind_{\mu,\mu'}M\otimes N$ with subquotients over all $\lambda\in \N[I]$ such that $\nu-\lambda, \mu'-\lambda, \nu'+\lambda-\mu' \in\N[I]$, which are evenly isomorphic to
$$\Pi^{p(\lambda)p(\nu'+\lambda-\mu')}\Ind^{\nu,\nu'}_{\nu-\lambda,\lambda,\nu'+\lambda-\mu',\mu'-\lambda} {}^\diamond(\Res^{\mu,\mu'}_{\nu-\lambda,\nu'+\lambda-\mu',\lambda,\mu'-\lambda}M\otimes N )\{-\lambda\cdot(\nu'+\lambda-\mu')\}.$$
Here if $\Res M\otimes N=Q_1\otimes Q_2\otimes Q_3\otimes Q_4$, then ${}^\diamond(\Res M\otimes N)=Q_1\otimes Q_3\otimes Q_2\otimes Q_4$.
}
\end{proposition}

\subsection{Categorification of quantum Serre relations and bilinear forms}\

\vskip 2mm

For $\ii\in\Seq(\nu)$,  we set
$$P_{\ii}=R(\nu)1_{\ii},$$
which is a  projective  graded $R(\nu)$-supermodule. Recall that for $i\in I^\re$, we define in (\ref{peven}) and (\ref{podd}) the projective graded supermodule
$$P_{i^{(n)}}= R(ni)e_{i,n}\left\{\binom{n}{2} \cdot r_i\right\}.$$

\begin{proposition}\label{Serre}
{\it Assume $i\in I^\re$, $j\in I$ and $i\neq j$. Let $n\in \Z_{>0}$ and $m=1-na_{ij}$. We have an evenly isomorphism of graded $R(mi+nj)$-supermodules
$$\bigoplus^{\lfloor \frac{m}{2} \rfloor}_{c=0}\Pi^{p(2c;i,j;n)}P_{i^{(2c)}j^ni^{(m-2c)}}\simeq\bigoplus^{\lfloor \frac{m-1}{2} \rfloor}_{c=0}\Pi^{p(2c+1;i,j;n)}P_{i^{(2c+1)}j^ni^{(m-2c-1)}}.$$
Here, $P_{i^{(a)}j^n i^{(b)}}=\Ind (P_{i^{(a)}}\otimes P_{j^n}\otimes P_{i^{(b)}})$.
Moreover, for $i,j\in I$ and $i\cdot j=0$, we have an evenly isomorphism
$$P_{ij}\simeq \Pi^{p(i)p(j)}P_{ji}.$$}
\begin{proof}
The proof is identical to the "Box" calculations in \cite{KL2011} and \cite{HW2015}. We leave the details to the reader.
\end{proof}
\end{proposition}

Let $\psi:R(\nu)\rightarrow R(\nu)$ be the anti-involution of $R(\nu)$ by flipping the diagrams about horizontal axis. For $P\in R(\nu)$-$\pMod$, define $\overline{P}=\HOM_{R(\nu)}(P,R(\nu))^\psi$ to be the left graded $R(\nu)$-supermodule with the action twisted by $\psi$. We have the following isomorphism for all $k, r\in\Z$:
$$\overline{\Pi^r P\{k\}}\simeq \Pi^r  \overline P\{-k\}.$$

Define the $\Z^\pi[q,q^{-1}]$-bilinear pairing
$( \ , \ ):K_0(R(\nu))\times G_0(R(\nu))\rightarrow \Z^\pi[q,q^{-1}]$ by
$$([P],[M])=\gdim (P^\psi\otimes_{R(\nu)}M)=\gdim \HOM_{R(\nu)}(\overline{P},M).$$
Since all irreducible graded $R(\nu)$-supermodules are of type $\mathtt M$, the groups $G_0(R(\nu))$ and $K_0(R(\nu))$ are dual $\Z^\pi[q,q^{-1}]$-module under this pairing.

  There is also a symmetric $\Z^\pi[q,q^{-1}]$-bilinear form
  $$( \ , \ ):K_0(R(\nu))\times K_0(R(\nu))\rightarrow \Z^\pi ((q))$$
defined in the same manner. Similar to \cite[Proposition 3.3]{KL2009}, we have the following proposition.

 \begin{proposition}\label{Sy} {\it The  bilinear form  on $K_0(R)$  satisfies:
 \begin{itemize}
\item[(1)] $([P],[Q])=0$  \ for $P\in R(\nu)\text{-}\pMod$, $Q\in R(\mu)\text{-}\pMod$ with $\nu\neq \mu$,
\item[(2)] $(1,1) = 1$, where $1=\C$ as a module over $R(0)=\C$,
\item[(3)] $([P_{i}], [P_{i}]) =(1-\pi^{p(i)}q_i^2)^{-1}$ for all $i\in I$,
\item[(4)] $(x, yz) = (\Res(x), y \otimes z)$  \ for $x,y,z
\in K_0(R)$.
\end{itemize}}
\end{proposition}

\subsection{Theorem of categorification}\

\vskip 2mm

Let $K_0(R)_{\Q^\pi(q)}=\Q^\pi(q)\otimes_{\Z^\pi[q,q^{-1}]}K_0(R)$. By (\ref{serrepp}) and Proposition \ref{Serre}, we have a well-defined bialgebra homomorphism
$$\Gamma: \HH\rightarrow K_0(R)_{\Q^\pi(q)}$$
given by $\Gamma (\theta_i)=[P_i]$ for all $i\in I$. By Proposition \ref{Sy}, the bilinear form $\{ \ , \ \}_\pi$ on $\HH$ and the form $( \ , \ )$ on $K_0(R)_{\Q^\pi(q)}$ take same values under $\Gamma$, i.e.,
$$(\Gamma(x),\Gamma (y))=\{x,y\}_\pi \ \ \text{for} \ x,y\in \HH.$$
Thus $\Gamma $ is injective by the non-degeneracy of $\{ \ , \ \}_\pi$ (see Proposition \ref{nondeg}). It induces an injective $\Z^\pi[q,q^{-1}]$-algebra homomorphism
$$\Gamma_{\mathcal A^\pi}:{_{\mathcal A^\pi} \HH}\rightarrow K_0(R).$$
Furthermore, $\Gamma_{\mathcal A^\pi}$ induces a $\Z[q,q^{-1}]$-algebra homomorphism
$$\Gamma_{\mathcal A^\pi}|_{\pi=1}:{_{\mathcal A^\pi} \HH}\big/(\pi-1)\rightarrow K_0(R)\big/(\pi-1)=K_0(R)\big/([\Pi P]-[P]),$$
where the left-hand side coincides with
${_{\mathcal A}\XX}$ by (\ref{A-form}), while the right-hand side coincides with
 $K_0(|R|)$ by Remark \ref{type}. Using a similar argument given in \cite{TW2023}, one can establish an isomorphism between ${_{\mathcal A}\XX}$ and $K_0(|R|)$, and  $\Gamma_{\mathcal A^\pi}|_{\pi=1}$ is precisely the same map obtained in that way.

Thus $\Gamma_{\mathcal A^\pi}|_{\pi=1}$ is an isomorphism, which implies the surjectivity of $\Gamma_{\mathcal A^\pi}$. We have shown the following:

\begin{theorem}\label{them:iso-}
{\it We have an isomorphism of $\Z^\pi[q,q^{-1}]$-bialgebras
$$\begin{aligned}
& \Gamma_{\mathcal A^\pi}:{_{\mathcal A^\pi} \HH}\xrightarrow{\sim} K_0(R) \\
&\phantom{\Gamma_{\mathcal A^\pi}:}\ \ \theta_i^{(n)}\mapsto [P_{i^{(n)}}]\qquad \text{for}\ i\in I^\re, n\in \Z_{> 0},\\
& \phantom{\Gamma_{\mathcal A^\pi}:}\ \ \ \theta_i\   \mapsto [P_i]\qquad \text{for}\ i\in I^\im.
\end{aligned}$$
Moreover, the specialization $\pi\mapsto -1$ induces an isomorphism
$$\Gamma_{\mathcal A^\pi}|_{\pi=-1}:{_{\mathcal A}\YY}  \left(\cong {_{\mathcal A^\pi} \HH}\big/(\pi+1)\right)\xrightarrow{\sim} K_0(R)\big/(\pi+1).$$
}
\end{theorem}

\vskip 5mm

\section{\textbf{Appendix}}

The following is parallel to \cite[Section 3]{Kas91}, which establishes the existence of the coproduct $\rho_-$ and the bilinear form $\{ \ , \ \}_-$ on the quantum Borcherds superalgebra $\YY$.

For $i\in I^{\tx{re}},n\geq k\geq 0$, define
$${\begin{bmatrix} n \\ k \end{bmatrix}}_i^-=\frac{[n]_i^-!}{[k]_i^-![n-k]_i^-!}.$$
Note that
$$[n]_i^-=(-1)^{\frac{n-1}{2}p(i)}[n]_{(-1)^{p(i)/2}q_i},\ \ [n]_i^- !=(-1)^{\frac{n(n-1)}{4}p(i)}[n]_{(-1)^{p(i)/2}q_i}!.$$
Here we use the notation $[k]_v=\frac{v^k-v^{-k}}{v-v^{-1}}$ for any integer $k$. 
By the binomial formula for $[k]_v$, one obtains
\begin{equation*}
{\begin{bmatrix} n+1 \\ k \end{bmatrix}}_i^-=q_i^{-k}{\begin{bmatrix} n \\ k \end{bmatrix}}_i^-+(-1)^{(n-k+1)p(i)}q_i^{n-k+1}{\begin{bmatrix} n \\ k-1 \end{bmatrix}}_i^-.
\end{equation*}

Given a Borcherds--Cartan super datum $(I,\widetilde A,\cdot)$, define the dual weight lattice
$$P^\vee=\Big(\bigoplus_{i\in I}\mathbb Z h_i\Big)\oplus \Big(\bigoplus_{i\in I}\mathbb Z d_i\Big)$$
and the Cartan subalgebra $\mathfrak h=\mathbb Q\otimes_{\mathbb Z}P^\vee$.
The weight lattice is defined to be
$$P=\{\lambda\in\mathfrak h^*\mid \lambda(P^\vee)\subset\mathbb Z\}.$$
For each $i\in I$, define the simple root $\alpha_i\in P$ by
$\alpha_j(h_i)=a_{ij}$ and $\alpha_j(d_i)=\delta_{ij}$. We identify $i\in I$ with $\alpha_i$, so that the positive root lattice
$Q_+=\bigoplus_{i\in I}\mathbb N\alpha_i$   is identified with $\N[I]$.
The set of dominant weights is
$$P^+=\{\Lambda\in P\mid \Lambda(h_i)\in \Z_{\geq 0}\  \tx{for}\ i\in I;\ \Lambda(h_i)\in 2\Z_{\geq 0} \ \tx{for}\ i\in I^{\tx{re}}\cap I_{\od} \}.$$

\begin{definition}\cite[Definition 2.7]{BKM1998}\  The {\it quantum Borcherds superalgebra} $U$ associated to  $(I,\widetilde A,\cdot)$ is the $\Q(q)$-algebra generated by $e_i,f_i$ $(i\in I)$ and $q^h$ $(h\in P^\vee)$, subject to
\begin{equation*}
\begin{aligned}
& q^0= 1,\quad q^hq^{h'}=q^{h+h'} \quad \text{for} \ h,h' \in P^{\vee}, \\
& q^h e_i q^{-h} = q^{\alpha_i(h)} e_i, \ \ q^h f_iq^{-h} = q^{-\alpha_i(h)} f_i\quad \text{for} \ i\in I, h \in P^{\vee}, \\
& \sum_{a+b=1-a_{ij}}(-1)^{a+p(a;i,j)}{\begin{bmatrix} 1-a_{ij} \\ a \end{bmatrix}}_i^- f_i^{a}f_jf_i^{b}=0 \quad \text{for} \ i\in I^\re,j\in I\ \text{and}\ i\neq j,\\
& \sum_{a+b=1-a_{ij}}(-1)^{a+p(a;i,j)}{\begin{bmatrix} 1-a_{ij} \\ a \end{bmatrix}}_i^- e_i^{a}e_je_i^{b}=0 \quad \text{for} \ i\in I^\re,j\in I\ \text{and}\ i\neq j,\\
& f_if_j-(-1)^{p(i)p(j)}f_jf_i= e_ie_j-(-1)^{p(i)p(j)}e_je_i= 0 \quad \text{for}\ i,j\in I\ \text{with}\ a_{ij}=0,\\
& e_if_j-(-1)^{p(i)p(j)}f_je_i= \delta_{ij}\frac{K_i-K_i^{-1}}{q_i-q_i^{-1}} \quad \text{for}\ i,j\in I,\   \tx{where}\ K_i=q_i^{h_i}.
\end{aligned}
\end{equation*}
\end{definition}

Let $U^+$ (resp. $U^-$, resp. $U^0$) be the subalgebra of $U$ generated by $e_i$ (resp. $f_i$, resp. $q^h$).  By \cite[Theorem 2.23]{BKM1998}, we have the triangular decomposition $U\cong U^+\otimes U^0\otimes U^-$. Moreover, $U^-$ coincides with $\YY$.

For any $i\in I, x\in U^-$, there exist unique $e_i'(x),e_i''(x)\in U^-$ such that
$$e_ix-(-1)^{p(i)p(x)}xe_i=\frac{K_ie_i''(x)-K_i^{-1}e_i'(x)}{q_i-q_i^{-1}}.$$
The operators $e_i', e_i'', f_i$ (which represent left multiplication on  $U^-$) satisfy the following relations for all $i,j\in I$:
\begin{equation}\label{B1}
e_i' f_j = \delta_{ij} + (-1)^{p(i)p(j)} q_i^{-a_{ij}} f_j e_i',
\end{equation}
\begin{equation}\label{B2}
e_i'' f_j = \delta_{ij} + (-1)^{p(i)p(j)} q_i^{a_{ij}} f_j e_i'',
\end{equation}

\begin{equation}\label{B3}
e_i' e_j''= (-1)^{p(i)p(j)} q_i^{a_{ij}} e_j'' e_i'.
\end{equation}

\begin{definition} \  The {\it quantum Boson superalgebra} $\mathscr B$ associated to  $(I,\widetilde A,\cdot)$ is the $\Q(q)$-algebra generated by $e_i', f_i$ $(i\in I)$, subject to
\begin{equation*}
\begin{aligned}
&  e_i'f_j=\delta_{ij}+(-1)^{p(i)p(j)}q_i^{-a_{ij}}f_je_i', \quad \text{for} \ i,j\in I, \\
& f_if_j-(-1)^{p(i)p(j)}f_jf_i= e_i'e_j'-(-1)^{p(i)p(j)}e_j'e_i'= 0 \quad \text{for}\ i,j\in I\ \text{with}\ a_{ij}=0,  \\
& \sum_{a+b=1-a_{ij}}(-1)^{a+p(a;i,j)}{\begin{bmatrix} 1-a_{ij} \\ a \end{bmatrix}}_i^- f_i^{a}f_jf_i^{b}=0 \quad \text{for} \ i\in I^\re,j\in I\ \text{and}\ i\neq j,\\
& \sum_{a+b=1-a_{ij}}(-1)^{a+p(a;i,j)}{\begin{bmatrix} 1-a_{ij} \\ a \end{bmatrix}}_i^- e_i'^{a}e_j'e_i'^{b}=0 \quad \text{for} \ i\in I^\re,j\in I\ \text{and}\ i\neq j.\\
\end{aligned}
\end{equation*}
\end{definition}

\begin{lemma}\cite[Lemma 3.4.2, 3.4.3]{Kas91} 
{\it \ $U^-$ is a left $\mathscr B$-module, and we have a canonical isomorphism of $\mathscr B$-modules:
$$\mathscr B\big/ \sum_{i\in I}  \mathscr B e_i'\cong  U^-.$$}
\begin{proof}
We check the defining relations of $\mathscr B$ in the endomorphism ring of $U^-$. First assume that $i,j\in I$ with $a_{ij}=0$. By equation (\ref{B1}), for any $k\in I$, we have
$$\begin{aligned}
&  e_i'e_j'f_k=\delta_{jk}e_i'+(-1)^{p(i)p(j)}\delta_{ik}e_j'+(-1)^{(p(i)+p(j))p(k)}q_k^{-a_{ki}-a_{kj}}f_ke_i'e_j', \\
& e_j'e_i'f_k=(-1)^{p(i)p(j)}\delta_{jk}e_i'+\delta_{ik}e_j'+(-1)^{(p(i)+p(j))p(k)}q_k^{-a_{ki}-a_{kj}}f_ke_j'e_i'.
\end{aligned}$$
Let $S=e_i'e_j'-(-1)^{p(i)p(j)}e_j'e_i'$. Then for each $k\in I$, $$Sf_k=(-1)^{(p(i)+p(j))p(k)}q_k^{-a_{ki}-a_{kj}}f_kS,$$
which implies that $S$ acts as zero on $U^-$.

Next assume $i\in I^{\tx{re}}, j\in I$ with $j\neq i$, and let $m=1-a_{ij}$. For any $k\in I, n\in\N$, we obtain from (\ref{B1})
$$e_i'^nf_k=(-1)^{(n-1)p(i)}q_i^{1-n}[n]_i^-\delta_{ik}e_i'^{n-1}+(-1)^{np(i)p(k)}q_i^{-na_{ik}}f_ke_i'^n.$$
Therefore, for any $a,b\in\N$ with $a+b=m$, we have
\begin{equation*}\begin{aligned}
& e_i'^{a}e_j'e_i'^{b}f_k =(-1)^{(mp(i)+p(j))p(k)}q_k^{-ma_{ki}-a_{kj}}f_ke_i'^{a}e_j'e_i'^{b}\\
& \phantom{e_i'^{a}e_j'e_i'^{b}f_k =}+(-1)^{bp(i)p(j)}q_i^{b(m-1)}\delta_{jk}e_i'^m\\
& \phantom{e_i'^{a}e_j'e_i'^{b}f_k =}+(-1)^{(b-1)p(i)}q_i^{1-b}[b]_i^-\delta_{ik}e_i'^{a}e_j'e_i'^{b-1}\\
& \phantom{e_i'^{a}e_j'e_i'^{b}f_k =}+(-1)^{(m-1)p(i)+p(i)p(j)}q_i^{-b}[a]_i^-\delta_{ik}e_i'^{a-1}e_j'e_i'^{b}.
\end{aligned}
\end{equation*}

Since $$\sum_{a+b=m}(-1)^av^{b(m-1)}{\begin{bmatrix} m \\ a \end{bmatrix}}_v=0,$$
 by setting $v=(-1)^{p(i)/2}q_i$, we obtain
$$\sum_{a+b=m}(-1)^{a+\frac{b(b-1)}{2}p(i)}q_i^{b(m-1)}{\begin{bmatrix} m \\ a \end{bmatrix}}_i^-=0.$$
By combining this with the fact that $\frac{a(a-1)}{2}+\frac{b(b-1)}{2}\equiv \frac{m-1}{2} \pmod{2}$ when $i$ is odd, we get
\begin{equation}\label{B3}\sum_{a+b=m}(-1)^{a+p(a;i,j)+bp(i)p(j)}q_i^{b(m-1)}{\begin{bmatrix} m \\ a \end{bmatrix}}_i^-=0.\end{equation}
Note also that
\begin{equation}\label{B4}p(a;i,j)+(b-1)p(i)+p(a+1;i,j)+(m-1)p(i)+p(i)p(j)\equiv 0\pmod{2}.\end{equation}
Now, define $$S= \sum_{a+b=m}(-1)^{a+p(a;i,j)}{\begin{bmatrix} m \\ a \end{bmatrix}}_i^- e_i'^{a}e_j'e_i'^{b}.$$
Using the identities (\ref{B3}) and (\ref{B4}), we obtain
$$Sf_k=(-1)^{(mp(i)+p(j))p(k)}q_k^{-ma_{ki}-a_{kj}}f_kS,$$
which implies that $S$ acts as zero on $U^-$.

The second part of the proposition follows immediately from \cite[Lemma 3.4.3]{Kas91}.
\end{proof}
\end{lemma}

\begin{lemma}\cite[Lemma 3.4.7]{Kas91} \label{BI}
{\it \ Let $x\in U^-$ satisfy $e_i'(x)=0$ for all $i\in I$. Then $x\in\Q(q)$.}
\begin{proof}
We prove the lemma by induction on $\text{ht}{(\nu)} \geq 1$. The case for $\text{ht}{(\nu)} = 1$ is trivial. Now assume $\text{ht}{(\nu)} > 1$. Let $x \in U^-_\nu$ such that $e_i'(x) = 0$ for all $i \in I$.

For any $i, j \in I$, we have by equation (\ref{B2}),
 $$e_i'e_j''(x)=(-1)^{p(i)p(j)} q_i^{a_{ij}} e_j'' e_i'(x)=0.$$
By the induction hypothesis, this implies that $e_j''(x) = 0$ for all $j \in I$. Thus, for any $j \in I$, have the commutation relation 
$$e_jx=(-1)^{p(i)p(j)}xe_j.$$
 
 By \cite[Corollary 4.10]{BKM1998}, the irreducible highest weight module $V(\Lambda)$ with highest weight $\Lambda\in P^+$ and highest weight vector $v_\Lambda$ is given by
$$
 V(\Lambda) \cong U^- \bigg/
(\sum_{i\in I^{\text{re}}}U^- f_i^{\Lambda(h_i)+1}+
\sum_{i\in I^\im \ \text{with}\ \Lambda(h_i)=0}U^-f_i).
$$
Choose $\Lambda \in P^+$ such that $\Lambda \gg 0$. Then, there is an isomorphism $$U_\nu^-\xrightarrow{\sim} V(\Lambda)_{\Lambda-\nu},\ u\mapsto uv_\Lambda.$$ For all $j \in I$, we have $e_j x v_\Lambda = 0$, which implies that the submodule $U x v_\Lambda$ of $V(\Lambda)$ does not contain $v_\Lambda$. Since $V(\Lambda)$ is irreducible, this forces $x v_\Lambda = 0$, implying that $x = 0$.
\end{proof}
\end{lemma}

\begin{theorem}\cite[Proposition 3.4.4, Corollary 3.4.8.]{Kas91} \label{BII}
{\it For any family $\{\kappa_i\}_{i\in I}$ of non-zero elements in $\Q(q)$, there exists a  nondegenerate symmetric bilinear form $\langle \ , \ \rangle: U^-\times U^-\rightarrow \Q(q)$ such that $\langle 1,1\rangle=1$ and
$$\langle f_ix,y\rangle=\kappa_i\langle x,e_i'y\rangle\ \ \tx{for all}\ i\in I.$$}
\begin{proof}
The assignment $f_i\mapsto \kappa_ie_i', e_i'\mapsto \kappa_i^{-1}f_i$ is a $\Q(q)$-algebra antiinvolution of $\mathscr B$. Thus $(U^-)^*$ is a left $\mathscr B$-module with the following action
$$(f_i\varphi)(y)=\kappa_i\varphi(e_i'y),\ (e_i'\varphi)(y)=\kappa_i^{-1}\varphi(f_iy)\ \ \tx{for}\ \varphi\in  (U^-)^*. $$

Let $\varphi_0\in (U^-)^*$ defined by $$\varphi_0(1)=1,\ \varphi_0(\sum_i f_iU^-)=0.$$
Thus  $e_i' \varphi_0 = 0$ for all $i \in I$. This gives rise to a $\mathscr{B}$-linear map
$$\mu: U^-\cong\mathscr B\big/ \sum_{i\in I}  \mathscr B e_i'\rightarrow (U^-)^*, \ 1\mapsto \varphi_0.$$

Now, define the bilinear form $\langle x, y\rangle := (\mu(x))(y)$. We have by the definition
$$\langle f_ix,y \rangle=\kappa_i\langle x,e_i'y\rangle,\quad \langle e_i'x,y\rangle=\kappa_i^{-1}\langle x,f_iy\rangle.$$
Such a bilinear form is unique. Since $\langle x,y\rangle':=\langle y,x\rangle$ satisfies the same property, we see that $\langle\ ,\ \rangle$ is symmetric.

To prove the non-degeneracy of $\langle x, y\rangle$, we use induction on $\text{ht}{(\nu)} \geq 1$. Let $x \in U^-_\nu$ such that $\langle x, U^-_{\nu}\rangle = 0$. Then, for all $i \in I$, we have $\langle e_i' x, U^-_{\nu - i} \rangle = 0$. By the induction hypothesis, this implies that $e_i' x = 0$ for all $i \in I$. By Lemma \ref{BI}, we conclude that $x = 0$.
\end{proof}
\end{theorem}

Let $\mathscr Y$ be the free $\Q(q)$-algebra generated by $\vartheta_i$ $(i\in I)$ with $|\vartheta_i|=i$ and $p(\vartheta_i)=p(i)$. We define a comultiplication $\rho_-:\mathscr Y\rightarrow \mathscr Y\otimes \mathscr Y$  by $\rho_-{(\vartheta_i)}=\vartheta_i\otimes 1+1\otimes \vartheta_i$ ($i\in I$). Here $\mathscr Y\otimes \mathscr Y$ is endowed with the twisted multiplication
$$(x_1\otimes x_2)(y_1\otimes y_2)=(-1)^{p(x_2)p(y_1)}q^{-|x_2|\cdot |y_1|}x_1y_1\otimes x_2y_2.$$ As in \cite[1.2.3]{Lus}, there is a   symmetric bilinear form $\{ \ , \ \}_-:\mathscr Y\times \mathscr Y\rightarrow \Q(q)$ satisfying
\begin{itemize}
\item[(i)] $\{x, y\}_- =0$ if $|x| \neq |y|$,
\item[(ii)] $\{1,1\}_-= 1$,
\item[(iii)] $\{\vartheta_i, \vartheta_i\}_- = \kappa_i$ for all $i\in I$,
\item[(iv)] $\{x, yz\}_-= \{\rho_-(x), y \otimes z\}_-$  for $x,y,z
\in \mathscr Y$.
\end{itemize}

Let $i\in I$. For $x\in \mathscr Y$, we write $\rho_-(x)$ into
$$\rho_-(x)=\vartheta_i\otimes \rho_-^i(x) + \ \text{terms of bidegree not in}\ i\times \N[I].$$
Then for any $x,y\in \mathscr Y$, we have by the definition
\begin{equation}\label{B5}\rho_-^i(xy) = \rho_-^i(x) + (-1)^{p(i)p(x)}q^{-i\cdot |x|}x\rho_-^i(y),\end{equation}
\begin{equation}\label{B6}\{\vartheta_ix, y\}_-=\kappa_i\{x,\rho_-^i(y)\}_-.\end{equation}

Following \cite[1.4.5]{Lus}, one can show that $\tx{rad}\{ \ , \ \}_-$ contains
$$
\begin{aligned}
& \sum_{a+b=1-a_{ij}}(-1)^{a+p(a;i,j)}\vartheta_i^{(a)}\vartheta_j\vartheta_i^{(b)} \quad \text{for} \ i\in I^\re,j\in I\ \text{and}\ i\neq j,\\
& \vartheta_i\vartheta_j-(-1)^{p(i)p(j)}\vartheta_j\vartheta_i \quad \text{for}\ i,j\in I\ \text{with}\ i\cdot j=0.
\end{aligned}
$$
Thus, we have a surjective homomorphism
$$\Phi: U^-\twoheadrightarrow \mathscr Y/\tx{rad}\{ \ , \ \}_-,\ f_i\mapsto \vartheta_i\ \ \tx{for}\ i\in I.$$

Note that $\rho_-$ and $\rho_-^i$ are well-defined on $\mathscr Y/\tx{rad}\{ \ , \ \}_-$. By (\ref{B5}), we see that
$$\Phi(e_i'x)=\rho_-^i\Phi(x) \ \ \tx{for all}\ i\in I,$$
and then by (\ref{B6}), we have
$$\{\Phi(x), \Phi(y)\}_-=\langle x,y\rangle \ \ \tx{for}\ x,y\in U^-.$$
Therefore, $\Phi$ is injective by the non-degeneracy of $\langle x, y\rangle$ (see Theorem \ref{BII}). 

We conclude that
$\Phi$ is an isomorphism, and so $\mathbf Y\cong \mathscr Y/\tx{rad}\{ \ , \ \}_-$. 
Thus, $\rho_-$ defines a comultiplication on $\mathbf Y$.

\vskip 6mm
\section*{Acknowledgement}
\vskip 2mm

The author would like to thank Bolun Tong for his valuable sugggestions on this thesis.
\vskip 10mm

\bibliographystyle{amsplain}

\end{document}